%% file: schur.tex
\documentclass[12pt]{article}

\usepackage{graphicx}
\usepackage{amsmath}
\usepackage{mathtools}
\usepackage{amssymb}
\usepackage{latexsym}
\usepackage{subfigure}
\usepackage{crop}
\usepackage{algorithmic}
\usepackage{algorithm}
\usepackage{multirow}
\usepackage[section,subsection,subsubsection]{placeins}
\usepackage[colorlinks = true, pdfstartview = FitV, linkcolor = blue, citecolor = blue, urlcolor = blue]{hyperref}
\usepackage{bm}
\usepackage{bbm}
\usepackage{enumerate}

\newcommand {\veps} {\varepsilon}

\newcommand {\vv}  { {\bf v} }
\newcommand {\ww}  { {\bf w} }

\newcommand{\hf}{\frac12}

\newcommand{\defeq}{\mathrel{\mathop:}=}

\newtheorem{theorem}{Theorem}

\newtheorem{corollary}[theorem]{Corollary}

\newtheorem{lemma}[theorem]{Lemma}

\newenvironment{proof}[1][Proof]{\begin{trivlist}
\item[\hskip \labelsep {\bfseries #1}]}{\end{trivlist}}

\newcommand{\qed}{\nobreak \ifvmode \relax \else
      \ifdim\lastskip<1.5em \hskip-\lastskip
      \hskip1.5em plus0em minus0.5em \fi \nobreak
      \vrule height0.75em width0.5em depth0.25em\fi}
      
\begin{document}

\title{Schur properties of convolutions of gamma random variables}
\author{Farbod Roosta-Khorasani\thanks{Dept. of Computer Science, University of British Columbia, Vancouver, Canada.
{\tt farbod@cs.ubc.ca}.} \and G\'{a}bor J. Sz\'{e}kely\thanks{National Science Foundation, Arlington, Virginia. {\tt gszekely@nsf.gov} and Alfr\'{e}d R\'{e}nyi Institute of Mathematics, Hungarian Academy of Sciences,  Budapest, Hungary.}}
\maketitle
\begin{abstract}
Sufficient conditions for comparing the convolutions of heterogeneous gamma random 
variables in terms of the usual stochastic order are established. 
Such comparisons are characterized by the Schur convexity properties of the cumulative distribution function of the convolutions.
Some examples of the practical applications of our results are given. 
\end{abstract}

\vspace{1pc}
\noindent
\textbf{Keywords}: Schur-convexity of  tails; majorization order; linear combinations; gamma 
distribution; tail probabilities.

\vspace{0.5pc} \noindent
\textbf{MSC 2010}: Primary 60E15; secondary 62E99

\section{Introduction and Main Result}

Linear combinations (i.e., convolutions) of independent gamma random variables (r.v's) often naturally arise in many applications in statistics, engineering, insurance, actuarial science and reliability. As such there has been extensive study of their stochastic properties in the literature. For examples of such  theoretical studies as well as applications see~\cite{kochar2010right, kochar2011tail, kochar2012some, furman2006tail, yu2011some, lihong2005stochastic, bon1999ordering, zhao2009mean, zhao2011some, merkle1994schur, amiri2011skewness, boland1994schur, szba, roszas} and references therein.

Bock et al.,~\cite{bock}, and Diaconis and Perlman,~\cite{diaconis1990bounds}, in their seminal works, first 
studied the Schur convexity properties of the cumulative distribution function of the linear combinations of independent gamma r.v's. Ever since, this topic and its variants have been studied by many researchers; see the references mentioned above. However despite all the efforts, the results in~\cite{bock} remained the best available, yet far from the best possible. Bakirov,~\cite{bakirov95} provided a tighter bound for the special case of convolutions of chi-squared r.v's of degree one, which often arise from quadratic forms. Here, we prove results regarding the Schur properties of the tails of convolutions of gamma r.v's with respect to the mixing weights and in terms of the usual stochastic order, and indeed sharpen some results given in~\cite{bock}. As a consequence, the result in~\cite{bakirov95} is also generalized. 

More specifically, let $X_{i},\; i=1,2,\ldots,n$, be $n$ independent and identically distributed (i.i.d) gamma distributed r.v's, parametrized by shape $\alpha > 0$ and rate $\beta > 0$ parameters with the probability density function (PDF)
\begin{equation*}
f(x) = \left\{
  \begin{array}{l l}
\frac{\beta^{\alpha}}{\Gamma(\alpha)} x^{\alpha - 1 } e^{-\beta x}  &\text{$x \ge 0$} \\
0 &\text{$x \leq 0$}
\end{array} \right..
\end{equation*}
Consider the following non-negative linear combinations of such r.v's
\begin{equation*}
\sum_{i=1}^{n}   \lambda_{i}  X_{i}, 
\end{equation*}
where $\lambda_{i} \geq 0, \; i=1,2,\ldots,n$, are real numbers. For a given $\alpha > 0$, $\beta > 0$ and $x > 0$, define 
\begin{equation}
P(\bm{\lambda};\alpha,\beta,x) \defeq \Pr \left( \sum_{i=1}^{n} \lambda_{i} X_{i} < x \right),
\label{P_x}
\end{equation}
where $\bm{\lambda} = (\lambda_{1},\lambda_{2},\ldots,\lambda_{n}) \in \mathbb{R}^{n}$. The aim of the present paper is to find the conditions allowing one to compare tail probabilities of the form~\eqref{P_x} with respect to the mixing weights $\bm{\lambda}$. In~\cite[Theorem 2.2]{roszas}, results regarding the extremal values of~\eqref{P_x} (i.e., maximal and minimal values with respect to $\bm{\lambda}$ and for given $\alpha,\beta$ and $x$) are proved. Here, we extend those results to be able to compare $\eqref{P_x}$ for any pair of weight vectors.

To that end, let us recall that the vector $\bm \lambda$ is said to \textit{majorize} the vector $\bm \mu$, denoted by $\bm \mu \prec \bm \lambda$, if 
\begin{subequations}
\begin{eqnarray}
&& 0 \leq \lambda_{n} \leq \ldots \leq \lambda_{2} \leq \lambda_{1}, \\
&& 0 < \mu_{n} \leq \ldots \leq \mu_{2} \leq \mu_{1}, \\
&& \sum_{i=1}^{k} \mu_{i} \leq \sum_{i=1}^{k} \lambda_{i}, \quad \forall k < n,\\
&& \sum_{i=1}^{n} \mu_{i} = \sum_{i=1}^{n} \lambda_{i}. 
\end{eqnarray}
\label{major}
\end{subequations}
Note that there is ``strict positivity'' assumption on $\bm{\mu}$, but not on $\bm{\lambda}$. Of course, it is clear that padding $\bm{\mu}$ (and as a result $\bm{\lambda}$) with 0's would not change the majorization order and simply add redundant components to both vectors.
Given $\bm \mu \prec \bm \lambda$, $P(.;\alpha,\beta,x)$ is said to be Schur convex if  $P(\bm{\mu};\alpha,\beta,x) \leq P(\bm{\lambda};\alpha,\beta,x)$, and it is said to be Schur concave if $P(\bm{\mu};\alpha,\beta,x) \geq P(\bm{\lambda};\alpha,\beta,x)$. For comprehensive details on the theory of majorization and its applications, refer to the classic book of Marshall and Olkin~\cite{marsholk79}. 

It is not hard to show that the variance of $\sum_{i=1}^{n}   \lambda_{i}  X_{i}$ is a Schur-convex function of $\bm{\lambda}$ and, indeed, it would be useful to know when $P(.;\alpha,\beta,x)$ exhibit similar properties. For symmetric distributions a fairly general result is known.  If $X_{1},X_{2}\ldots,X_{n}$ are independent random variables with a common
symmetric and log-concave PDF, in~\cite{proschan}, it was shown that  $\Pr \left( \sum_{i=1}^{n} \lambda_{i} X_{i} < x \right)$ is Schur-convex in $\bm{\lambda}$ for any $x > 0$. However, for positive random variables with non-symmetric distributions (such as gamma r.v's), to the best of our knowledge, no such general results, except for those in~\cite{bock} and~\cite{bakirov95}, exist.

For the case of $n \geq 3$, Bock et al.,~\cite[Theorem 3]{bock}, showed that if $\bm \mu \prec \bm \lambda$ and $\lambda_{i} > 0\; \forall i$, then 
\begin{eqnarray*}
P(\bm{\mu};\alpha,\beta,x) &\geq& P(\bm{\lambda};\alpha,\beta,x), \quad \forall x > \frac{(n \alpha + 1) \max_{i} \lambda_{i}}{\beta}, \\
P(\bm{\mu};\alpha,\beta,x) &\leq& P(\bm{\lambda};\alpha,\beta,x), \quad \forall x < \frac{(n \alpha + 1) \min_{i} \lambda_{i}}{\beta}.
\end{eqnarray*}
For the special case of chi-squared r.v's of degree one (i.e., $\alpha = \beta = 1/2$), Bakirov,~\cite{bakirov95}, provided the following tighter bound (for $\forall n \geq 1$) for the Schur concavity of $P(. ;1/2,1/2,x)$:
\begin{equation*}
P(\bm{\mu};\hf,\hf,x) \geq P(\bm{\lambda};\hf,\hf,x), \quad \forall x > 2 s,
\end{equation*}
where $s = \sum_{i=1}^{n} \lambda_{i} = \sum_{i=1}^{n} \mu_{i}$. No results concerning the Schur convexity of $P(. ;1/2,1/2,x)$ was given in~\cite{bakirov95}.

\input exline

Our main result is stated as follows (the details of its proof are given in the Appendix). The discussions regarding the relative improvements compared to~\cite{bock} and~\cite{bakirov95}, as well as further extensions are deferred to Section~\ref{sec:discuss}.
\begin{theorem}
Let $X_{i} \sim Gamma(\alpha,\beta),\; i=1,2,\ldots,n$, be $n$ i.i.d gamma r.v's, where $\alpha >0$ and $\beta > 0$. If $\bm \mu \prec \bm \lambda$, then 
\begin{eqnarray*}
P(\bm{\mu};\alpha,\beta,x) &\geq& P(\bm{\lambda};\alpha,\beta,x), \quad \forall x > \frac {(2  \alpha + 1) s}{2 \beta }, \\
P(\bm{\mu};\alpha,\beta,x) &\leq& P(\bm{\lambda};\alpha,\beta,x), \quad \forall x <
  \begin{dcases}
		\frac{\alpha s}{\beta},~& n = 2 \cr
		\frac{(\alpha-1) s}{\beta} ,~& n \geq 3, \; \alpha > 1 
	\end{dcases},
\end{eqnarray*}
where $s = \sum_{i=1}^{n} \lambda_{i} = \sum_{i=1}^{n} \mu_{i}$.
\label{majorization_thm}
\end{theorem}

This paper is organized as follows. In Section~\ref{sec:examples}, we give some examples of the practical applications of our results. In Section~\ref{sec:discuss}, we discuss the relative improvements of our results compared to those in~\cite{bock} and~\cite{bakirov95}. We also extend Theorem~\ref{majorization_thm} by weakening the majorization requirement. In addition, we give similar results for the case where $n = \infty$. The proofs of our results are given in the Appendix.

\section{Examples}
\label{sec:examples}
In this section, we give examples to demonstrate some practical applications of our results.
\subsection{Experimental Design in Signal Detection}
\label{ex:signal}
Consider the additive model of observations
\begin{equation*}
D(t) = \tau s(t) + \eta(t), \quad t \in [0, T]
\end{equation*}
where $D(t)$ is the measured data, $s(t)$ is the signal of interest and $\eta(t)$ is the additive noise. Suppose we have $N$ measurements, taken at discrete time intervals, $0 = t_{1} \leq t_{1} \leq \ldots \leq t_{N} = T$. In addition, suppose that $\{\eta(t_{i}): i = 1,2,\ldots,N\}$ is a collection of i.i.d Laplace r.v's with mean zero and variance $\sigma_{i}^{2}$. In addition, let $\tau = 1$ if there is a signal, and $\tau = 0$ otherwise (i.e., the measured data is in fact entirely the noise). To detect if the signal is present, we can use
the signal to noise ratio (SNR)
\begin{equation*}
Q(N) \defeq \frac{1}{N} \sum_{i=1}^{N} \frac{1}{\sigma_{i}}| D(t_{i})|.
\end{equation*}
Note that $\sigma_{i}^{-1} D(t)$ is a Laplace r.v, with mean $\tau s(t)$ and variance $1$. 

A design question to answer is that of , at least how many measurements, a priori, is needed to make sure the probability of Type I error (i.e., when we conclude that the signal is present, when in fact it is missing) is below a desired tolerance $0 < \delta \ll 1$.

If $\tau = 0$, then $\sigma_{i}^{-1} | D(t)|  \sim Gamma (1,\sqrt{2})$, so by Theorem~\ref{majorization_thm} we get
\begin{equation*}
P(Q(1) \geq x) \geq P(Q(2) \geq x) \geq \ldots \geq P(Q(N) \geq x) \geq \ldots ,\quad \forall x > \frac{3}{2\sqrt{2}}.
\end{equation*}
For a given $N$, we can easily compute $P(Q(N) \geq x)$. Hence, in order to find the minimum $N$ required, we can increase $N$ until $P(Q(N) \geq x) \leq \delta$.

\subsection{Matrix Trace Estimation}
\label{sec:trace}

The need to estimate the trace of an implicit symmetric positive semi-definite (SPSD) matrix
is of fundamental importance (see~\cite{sdr}) and arises in many applications; 
see for instance \cite{hutchinson,bafago,avto,HaberChungHermann2010,doas3,yori,rodoas1,rodoas2,learhe,gohewa,avron}
and references therein.
The standard approach for estimating the trace of such a matrix $A$, denoted here by $tr(A)$, is based on a Monte-Carlo method, where one generates $N$ 
random vector realizations $\ww_{i}$ from a suitable probability distribution $D$ and computes 
\begin{equation*}
tr_{D}^{N}(A) \defeq \frac{1}{N} \sum_{i=1}^{N} \ww_{i}^{t} A \ww_{i} .
\end{equation*}
One such suitable probability distribution is the standard normal, $\mathcal{N}(0,\mathbb{I})$. This estimator is known as the \textit{Gaussian estimator}, denoted here by $tr_{G}^{N}(A)$.

Now, given a pair of small positive real numbers $(\veps,\delta)$, 
consider finding an appropriate sample size $N$ such that
\begin{subequations}
\begin{eqnarray}
\Pr\Big( tr_{G}^{N}(A) \geq (1-\veps) tr(A) \Big) \geq 1-\delta \label{prob_ineq_lower}, \\
\Pr\Big( tr_{G}^{N}(A) \leq (1+\veps) tr(A) \Big) \geq 1-\delta \label{prob_ineq_upper}.
\end{eqnarray}
\label{prob_ineq_lower_upper}
\end{subequations}
Such question was first studied in~\cite{avto} and further improved in~\cite{roas1}. In particular, in~\cite{roas1} it was proved that the inequalities~\eqref{prob_ineq_lower_upper}
hold if
\begin{equation}
N > \frac{\| A \|}{tr(A)} \frac{8}{\veps^{2}} \ln (\frac{1}{\delta}),
\label{gauss_bd}
\end{equation}
where $\| A \|$ denotes the $L_{2}$ norm of the matrix $A$. The ratio $tr(A)/ \| A \|$ is known as the \textit{effective rank} of the matrix (see~\cite{eldar2012compressed}) and it is a stable quantity compared with the usual rank. 
The appearance of effective rank in the bound~\eqref{gauss_bd} is an indication of a possible relation between the ``\textit{skewness}'' of the eigenvalues of $A$ and the efficiency of the Gaussian estimator. In other words, the more skewed the eigenvalue distribution is, the worse we expect the Gaussian estimator to perform (i.e., the larger the true sample size required would be). However effective rank is not a consistent measure for skewness and, as such,
in~\cite{roas1}, this relationship was demonstrated only numerically and no consistent definition for how the relative skewness could be measured was given. Now using the majorization order among eigenvalue vectors as a consistent measure of skewness, the new theoretical results in the present paper fully  describe the observations from the numerical examples in~\cite{roas1}. As in the proof of~\cite[Theorem 1]{roas1}, we see that 
\begin{equation}
Pr\left( tr_{H}^{N}(A) \leq (1-\veps)tr(A) \right)  = P\left(\bm{\lambda};\frac{N}{2},\frac{N}{2},(1-\veps) tr(A)\right),
\end{equation}
where $\bm{\lambda}$ the vector of eigenvalues of $A$ sorted in the decreasing order.
%
Consider two SPSD matrices, $A_{1}$ and $A_{2}$, such that $tr(A_{1}) = tr(A_{2})$ and whose respective eigenvalue vectors, $\bm{\lambda}_{1}$ and $\bm{\lambda}_{2}$, are sorted in the decreasing order. If $\bm{\lambda}_{2} \prec \bm{\lambda}_{1}$, then we say that the eigenvalue distribution of $A_{1}$ is more skewed that that of $A_{2}$. If $N > 2/\veps$, then from Theorem~\ref{majorization_thm} we obtain
\begin{equation*}
P \left( \bm{\lambda}_{2};\frac{N}{2},\frac{N}{2},(1-\veps) tr(A_{2}) \right) \leq P \left( \bm{\lambda}_{1};\frac{N}{2},\frac{N}{2},(1-\veps) tr(A_{1}) \right).
\end{equation*}
In other words, using the same sample size $N$, our estimate with $A_{1}$ is more likely to be located further away to the left of the true value than that with $A_{2}$. Hence in order to make the former estimate better, we need to increase the sample size which, in turn, in some algorithms translates into more computational costs; see~\cite{HaberChungHermann2010,doas3, learhe,rodoas1,rodoas2,roszas}. 
Similar comparisons can be made for $Pr\left( tr_{H}^{N}(A) \leq (1+\veps)tr(A) \right)$. 

\section{Discussions and Further Extensions}
\label{sec:discuss}

The comparison between the results in Theorem~\ref{majorization_thm} and those in~\cite{bock} and~\cite{bakirov95} can be summarized as follows:
\begin{itemize}
\item For the special case of chi-squared distribution, i.e., $\alpha = \beta = 1/2$, our Schur concavity result is the same as that in~\cite{bakirov95}.
\item For $n=2$, the results of Theorem~\ref{majorization_thm} coincide with~\cite[Theorem 1]{bock}. However, for $n \geq 3$, our results are more uniform than~\cite[Theorem 3]{bock}. Namely, the sufficient conditions given in Theorem~\ref{majorization_thm} are not dependent on the dimensions of the vectors. Additionally, our bounds for $x$ are independent of the particular values of the vector components, such as ``$\min_{i} \lambda_{i}$'' or ``$\max_{i} \lambda_{i}$''. In other words, for all vectors whose sums are equal, we give a fixed bound for $x$. 
\item For $n \geq 3$, our Schur convexity result is sharper than that of~\cite[Theorem 3]{bock} when $$\max_{i} \lambda_{i} > \frac{\sum_{i} \lambda_{i}}{n}\frac{\alpha+\hf}{\alpha+1/n}.$$ Similarly, for $n \geq 3$ and  $\alpha > 1$, our Schur concavity result improves that of ~\cite[Theorem 3]{bock} when $$\min_{i} \lambda_{i} < \frac{\sum_{i} \lambda_{i}}{n}\frac{\alpha-1}{\alpha+1/n}.$$
\item Most importantly, in~\cite[Theorem 3]{bock}, it is required that $\lambda_{i} > 0,\; i = 1,2,\ldots,n$. This is a rather strong condition as many interesting comparisons cannot be performed this way. Our results do not make any strict positivity assumption on the components of $\bm{\lambda}$ and it suffices if they are simply non-negative. As a simple example, consider $\alpha = 2$, $\beta = 1$ and $x < 1$. Then using Theorem~\ref{majorization_thm}, we get $P(\bm{\lambda}_{1};2,1,x) \geq P(\bm{\lambda}_{2};2,1,x) \geq P(\bm{\lambda}_{3};2,1,x) \geq P(\bm{\lambda}_{4};2,1,x)$ with $\bm{\lambda}_{1} = (1,0,0,0)$, $\bm{\lambda}_{2} = (2/3,1/3,0,0)$, $\bm{\lambda}_{3} = (3/6,2/6,1/6,0)$, and $\bm{\lambda}_{4} = (4/10,3/10,2/10,1/10)$. This comparison is not possible with~\cite[Theorem 3]{bock}. 
\item For the Schur convexity of $P(.;\alpha,\beta,x)$ in the case of $\alpha \leq 1$ and $n \geq 3$, the result in~\cite[Theorem 3]{bock} remains the best available, as we were not able to improve upon it here. In fact, from our method of proof, it seems likely that obtaining a general result for this case is impossible. Indeed, Bock et al.\ \cite[p. 394]{bock} give an example which corroborates this observation.
\end{itemize}

It is possible to weaken the majorization requirement and obtain even more general results. More specifically, relaxing the equality condition in~\eqref{major} gives the following weak majorization order. Recall that the vector 
$\bm \lambda$ is said to \textit{weakly majorize} the vector $\bm \mu$, denoted by $\bm \mu \prec_{w} \bm \lambda$, if 
\begin{subequations}
\begin{eqnarray*}
&& 0 \leq \lambda_{n} \leq \ldots \leq \lambda_{2} \leq \lambda_{1}, \\
&& 0 < \mu_{n} \leq \ldots \leq \mu_{2} \leq \mu_{1}, \\
&& \sum_{i=1}^{k} \mu_{i} \leq \sum_{i=1}^{k} \lambda_{i}, \quad \forall k \leq n.
\end{eqnarray*}
\label{major_w}
\end{subequations}

In the case of weak majorization order, we have the following almost immediate corollary.
\begin{corollary}
Let $X_{i} \sim Gamma(\alpha,\beta),\; i=1,2,\ldots,n$, be $n$ i.i.d gamma r.v's, where $\alpha >0$ and $\beta > 0$. If $\bm \mu \prec_{w} \bm \lambda$, then 
\begin{eqnarray*}
P(\bm{\mu};\alpha,\beta,x) &\geq& P(\bm{\lambda};\alpha,\beta,x), \quad \forall x > \frac {(2  \alpha + 1) s_{\bm{\lambda}}}{2 \beta }, \\
P(\bm{\mu};\alpha,\beta,x) &\leq& P(\bm{\lambda};\alpha,\beta,x), \quad \forall x <
  \begin{dcases}
		\frac{\alpha s_{\bm{\mu}}}{\beta},~& n = 2 \cr
		\frac{(\alpha-1) s_{\bm{\mu}}}{\beta} ,~& n \geq 3, \; \alpha > 1 
	\end{dcases},
\end{eqnarray*}
where $s_{\bm{\lambda}} = \sum_{i=1}^{n} \lambda_{i}$ and $s_{\bm{\mu}} = \sum_{i=1}^{n} \mu_{i}$.
\label{majorization_cor_weak}
\end{corollary}

We can also extend Theorem~\ref{majorization_thm}, as well as Corollary~\ref{majorization_cor_weak}, for the case where $ n = \infty$. For any non-negative  $\ell_{1}$ sequence $\bm{\lambda} = (\lambda_{1}, \lambda_{1}, \ldots)$, i.e., $\sum_{i=1}^{\infty} \lambda_{i} < \infty$, define
\begin{equation*}
P_{\infty}(\bm{\lambda};\alpha,\beta,x) \defeq \Pr \left( \sum_{i=1}^{\infty} \lambda_{i} X_{i} < x \right).
\end{equation*}
The (weak) majorization order is naturally extended to such $\ell_{1}$ sequences.
\begin{corollary}
Let $\left \{X_{i} \sim Gamma(\alpha,\beta),\; i=1,2,\ldots \right \}$ be a countably infinite  collection of i.i.d gamma r.v's, where $\alpha >0$ and $\beta > 0$. For any two non-negative $\ell_{1}$ sequences, $\bm \mu$ and $\bm \lambda$, such that $\bm \mu \prec_{w} \bm \lambda$, we have 
\begin{eqnarray*}
P_{\infty}(\bm{\mu};\alpha,\beta,x) &\geq& P_{\infty}(\bm{\lambda};\alpha,\beta,x), \quad \forall x > \frac {(2  \alpha + 1) s_{\bm{\lambda}}}{2 \beta }, \\
P_{\infty}(\bm{\mu};\alpha,\beta,x) &\leq& P_{\infty}(\bm{\lambda};\alpha,\beta,x), \quad \forall x <
		\frac{(\alpha-1) s_{\bm{\mu}}}{\beta} , \quad \alpha > 1,
\end{eqnarray*}
where $s_{\bm{\lambda}} = \sum_{i=1}^{\infty} \lambda_{i}$ and $s_{\bm{\mu}} = \sum_{i=1}^{\infty} \mu_{i}$.
\label{majorization_cor_weak_infty}
\end{corollary}

\input exline
Finally, it might be worth noting that the results such as Theorem~\ref{majorization_thm} show that if $\bm{\mu} \prec \bm{\lambda}$, then $P(\bm{\mu};\alpha,\beta,x)$ and $P(\bm{\lambda};\alpha,\beta,x)$ must have at least one crossing on $x \in (0,\infty)$.  
Diaconis and Perlman in~\cite{diaconis1990bounds} tried to answer whether this crossing point is unique. However, they only proved this uniqueness for $n=2$ and for $n \geq 3$, they required to impose further restrictions. Ever since, this has been an open problem which is known as the Unique Crossing Conjecture (UCC) and it is quite remarkable that the UCC has remained open, although all the evidence points towards the direction of it being true.

\input exline
\noindent{\bf Acknowledgment}
We wish to thank Profs.\ Milan Merkle and Maochao Xu for their valuable comments during the preparation of the text.

\bibliographystyle{plain}
\bibliography{biblio}

\appendix
\section{Proof}
\label{proof}
In what follows $X \sim Gamma(\alpha,\beta)$ denotes a gamma r.v parametrized by shape $\alpha > 0$ and rate $\beta > 0$, $f_{X}$ and $F_{X}$ stand, respectively, for the probability density function (PDF) and cumulative distribution function (CDF) of a r.v $X$. Bold face letters denote vectors and the vector components are denoted as $\vv = (v_{1},v_{2},\ldots,v_{n})$. 

For the proof of Theorem~\ref{majorization_thm}, we need to make use of the following additional results. The proof of Theorem~\ref{theorem_unimodal} is identical to the proof of~\cite[Theorem 4]{szba}. Note that~\cite[Theorem 4]{szba} has been stated in terms of convolution of chi-squared r.v's but the proof, there, has been given for the more general case of arbitrary gamma r.v's. The details of the proofs for Lemmas~\ref{lemma_A*} and~\ref{lemma_2*} can be found in~\cite{roszas}. Lemma~\ref{lemma_1*} has been stated in~\cite{roszas} but the proof is omitted there. We give a detailed proof of Lemma~\ref{lemma_1*} here for completeness.

Theorem~\ref{theorem_unimodal} is essential in proving our results and it states that an arbitrary convolution of heterogeneous gamma random variables (not necessarily with a common shape or a common rate) has an unique mode. Recall that a PDF, $f(x)$, is called unimodal if there exists a unique $x = a$ such that $f(x)$ is non-decreasing for $x < a$ and $f(x)$ is non-increasing for $x > a$. The point $a$ is called the unique mode of $f(x)$.

\begin{theorem}
Let $X_{i} \sim	Gamma(\alpha_{i},\beta_{i}),\;  i=1,2,\ldots,n,$ be independent r.v's, where $\alpha_{i}, \beta_{i} > 0 \; \forall i$. The PDF of $Y_{n} \defeq \sum_{i=1}^{n} \lambda_{i} X_{i}$ is unimodal where $\lambda_{i} \geq 0 \; \forall i$.
\label{theorem_unimodal}
\end{theorem}

\begin{lemma}[{\cite[Lemma B.1]{roszas}}]
Let $X_{i} \sim	Gamma(\alpha_{i},\beta_{i}),\;  i=1,2,\ldots,n,$ be independent r.v's, where $\alpha_{i}, \beta_{i} > 0 \; \forall i$. Define $Y_{n} \defeq \sum_{i=1}^{n} \lambda_{i} X_{i}$ for $\lambda_{i} > 0$, $\forall i$ and $\rho_{j} \defeq \sum_{i=1}^{j} \alpha_{i}$.
Then for the PDF of $Y_{n}$, $f_{Y_{n}}$, we have
\begin{enumerate}[(i)]
	\item $f_{Y_{n}} > 0$, $\forall x > 0$,
	\item $f_{Y_{n}}$ is analytic on $\mathbb{R}^{+} = \{x | x > 0\}$,
	\item $f_{Y_{n}}^{(k)}(0) = 0$, if $0 \leq k < \rho_{n} - 1$, where $f_{Y_{n}}^{(k)}$ denotes the $k^{th}$ derivative of $f_{Y_{n}}$.
\end{enumerate} 
\label{lemma_A*}
\end{lemma}

\begin{lemma}[{\cite[Lemma B.2]{roszas}}]
Let $X_{i} \sim	Gamma(\alpha_{i},\alpha), \; i=1,2,\ldots,n,$ be independent r.v's, where $\alpha_{i} > 0 \; \forall i$ and $\alpha > 0$. Also let $\psi \sim Gamma(1,\alpha)$ be another r.v independent of all $X_{i}$'s. If $\sum_{i=1}^{n} \alpha_{i} > 1$, then the mode, $\bar{x}(\lambda)$, of the r.v 
$W({\lambda}) = Y + \lambda \psi$
is strictly increasing in $\lambda>0$, where $Y = \sum_{i=1}^{n} \lambda_{i} X_{i}$ with $\lambda_{i} > 0$, $\forall i$.
\label{lemma_1*}
\end{lemma}
\begin{proof}
By Lemma~\ref{lemma_A*},  $\bar{x}(\lambda) > 0$ for $\lambda \geq 0$. By the unimodality of $W({\lambda})$, for any $\lambda > \lambda_{0} > 0$, it is enough to show that 
\begin{equation}
J \big( \lambda,\bar{x}(\lambda_{0}) \big) \defeq \left[ \frac{d^{2}}{d x^{2}} \Pr \left( W({\lambda}) \leq x \right) \right]_{x = \bar{x}(\lambda_{0})} > 0.
\label{J_lemma_1*}
\end{equation}
Note that $J \big( \lambda_{0},\bar{x}(\lambda_{0}) \big) = 0$ and since $\sum_{i=1}^{n} \alpha_{i} > 1$, by Lemma~\ref{lemma_A*}(iii), $f_{Y} (0) = 0$. So we have
\begin{eqnarray*}
J \big( \lambda,\bar{x}(\lambda_{0}) \big)  &=& \left[ \frac{d}{d x} \int_{0}^{x} f_{Y}(x-z) \frac{\alpha}{\lambda} e^{-\frac{\alpha}{\lambda} z}  dz \right]_{x = \bar{x}(\lambda_{0})} \\
&=&  \left[ \int_{0}^{x} \frac{d}{d x} f_{Y}\big(x-z\big) \frac{\alpha}{\lambda} e^{-\frac{\alpha}{\lambda} z} dz \right]_{x = \bar{x}(\lambda_{0})}\\
&=& \int_{0}^{\bar{x}(\lambda_{0})} f^{'}_{Y}(z)  \frac{\alpha}{\lambda} e^{-\frac{\alpha}{\lambda} \big(\bar{x}(\lambda_{0})-z\big)} dz .
\end{eqnarray*}
Therefore,
\begin{equation*}
\int_{0}^{\bar{x}(\lambda_{0})} f^{'}_{Y}(z)  e^{\frac{\alpha z}{\lambda_{0}}} dz = \frac{\lambda_{0}}{\alpha} e^{\frac{\alpha \bar{x}(\lambda_{0})}{\lambda_{0}}} J \big( \lambda_{0},\bar{x}(\lambda_{0}) \big) = 0.
\end{equation*}
Thus for $\lambda > \lambda_{0} > 0$, we have
\begin{eqnarray*}
\frac{\lambda}{\alpha} e^{\frac{\alpha \bar{x}(\lambda)}{\lambda}} J \big( \lambda,\bar{x}(\lambda_{0}) \big) &=& \int_{0}^{\bar{x}(\lambda_{0})} f^{'}_{Y}(z)  e^{\frac{\alpha z}{\lambda}} dz \\
&=& \int_{0}^{\bar{x}(\lambda_{0})} f^{'}_{Y}(z)  e^{\frac{\alpha z}{\lambda}} - f^{'}_{Y}(z)  e^{\frac{\alpha z}{\lambda_{0}}} e^{\alpha \bar{x}(0)  \left(\frac{1}{\lambda} - \frac{1}{\lambda_{0}} \right)} dz \\
&=& \int_{0}^{\bar{x}(\lambda_{0})} f^{'}_{Y}(z)  \left( e^{\frac{\alpha z}{\lambda}} -  e^{\frac{\alpha z}{\lambda_{0}} + \alpha \bar{x}(0) \left(\frac{1}{\lambda} - \frac{1}{\lambda_{0}} \right) } \right) dz \\
&=& \int_{0}^{\bar{x}(\lambda_{0})} f^{'}_{Y}(z)  \left( e^{\frac{\alpha z}{\lambda}} -  e^{\frac{\alpha z}{\lambda} + \Phi \big(z,\bar{x}(0) \big) } \right) dz ,
\end{eqnarray*}
where $\bar{x}(0) > 0$ is the mode of r.v $Y$ and 
\begin{equation*}
\Phi \big(z,\bar{x}(0)\big) \defeq \alpha \Big(z - \bar{x}(0)\Big) \Big(\frac{1}{\lambda_{0}} - \frac{1}{\lambda} \Big) .
\end{equation*}
Now if $z < \bar{x}(0)$ then $\Phi \big(z,\bar{x}(0)\big) < 0$ and $f^{'}_{Y}(z) > 0$ so we get $J \big( \lambda,\bar{x}(\lambda_{0}) \big) > 0$. Similarly if $z > \bar{x}(0)$ then $\Phi \big(z,\bar{x}(0)\big) > 0$ and $f^{'}_{Y}(z) < 0$ and again we have $J \big( \lambda,\bar{x}(\lambda_{0}) \big) > 0$.
\end{proof}

\begin{lemma}[{\cite[Lemma B.3]{roszas}}]
For some $ \alpha_{2} \geq \alpha_{1} > 0$, let $\xi_{1} \sim Gamma(1+\alpha_{1},\alpha_{1})$ and $\xi_{2} \sim Gamma(1+\alpha_{2}, \alpha_{2})$ be independent gamma r.v's. Also let $\bar{x} = \bar{x}(\lambda)$ denote the mode of the r.v $\xi(\lambda) = \lambda \xi_{1}  + (1- \lambda) \xi_{2}$ for $0 \le \lambda \le 1$. Then $1 \le  \bar{x}(\lambda)   \le  \left(2  \sqrt{\alpha_{1}  \alpha_{2}} + 1\right)/\left(2 \sqrt{\alpha_{1}  \alpha_{2}} \right), \quad \forall 0 \le \lambda \le 1$, with $\bar{x}(0) = \bar{x}(1) = 1$ and, in case of $\alpha_{i} = \alpha_{j} = \alpha$, $\bar{x}(1/2) = \left(2 \alpha + 1\right)/\left(2 \alpha \right)$, otherwise the inequalities are strict.
\label{lemma_2*}
\end{lemma}

\subsection{Proof of Theorem~\ref{majorization_thm}}
\label{proof_majorization_thm} 
We prove the theorem for the case where $\alpha = \beta$ and $s=1$. The general case follows from the scaling properties of gamma distribution. 

We first consider the case where $n \geq 3$. If $\bm \lambda \succ \bm \mu$, then there exists a finite number, $r$, of vectors $\bm{\eta}_{i}, i = 1,2,\ldots r$, such that $\bm \lambda = \bm{\eta}_{1} \succ \bm{\eta}_{2} \succ \ldots \succ \bm{\eta} _{r-1} \succ \bm{\eta}_{r} = \bm{\mu}$, and $\bm{\eta}_{i}$ and $\bm{\eta}_{i+1}$ differ in two coordinates only, $i = 1,2,\ldots,r-1$, see~\cite[12.5.a]{peprto}. Thus we may, without loss of generality assume that $\bm \lambda$, and $\bm \mu$ differ only in two coordinates, and in fact, assume that for some $1\leq j < k \leq n$, we have
\begin{eqnarray}
(\lambda_{j},\lambda_{k}) \succ (\mu_{j},\mu_{k}) \; \text{ and } \; 
\lambda_{i} = \mu_{i} \text{ for } i \in \{1,\ldots,n\} \backslash \{j,k\}.
\label{shcur_vec_2_componts}
\end{eqnarray}
For $t \in [0,1]$, define
\begin{eqnarray}
\nu_{i}(t) &\defeq& t \lambda_{i} + (1-t) \mu_{i}, \quad i = j,k, \nonumber \\
\nu_{i}(t) &\defeq& \lambda_{i} , \quad i \neq j,k, \nonumber \\
Y(t) &\defeq& \sum_{i = 1}^{n} \nu_{i}(t) X_{i}.
\label{schur_vect}
\end{eqnarray}
It suffices to show that the CDF of $Y(t)$, in $t \in [0,1]$, is non-increasing for $x > (2  \alpha + 1)/(2 \alpha)$ and non-decreasing for $x < (\alpha - 1)/\alpha$ . Now we take the Laplace transform of $F_{Y(t)}$ as 
\begin{eqnarray*}
J(t,z) \defeq \mathcal{L}[F_{Y(t)}](z) &=& \int_{0}^{\infty} e^{-zx} F_{Y(t)}(x) dx \\
&=& \frac{-1}{z} \int_{0}^{\infty} F_{Y(t)}(x) d\left( e^{-zx} \right) \\
&=& \frac{1}{z} \int_{0}^{\infty} e^{-zx} dF_{Y(t)}(x) \\
&=& \frac{1}{z} \mathcal{L}[Y(t)](z),
\end{eqnarray*}
where $\mathcal{L}[Y(t)](z)$ is the Laplace transform of $Y(t)$ as
\begin{eqnarray*}
\mathcal{L}[Y(t)](z) = \prod \limits_{i =1}^{n} \left(1 + \frac {\nu_{i}(t) z}{\alpha}\right)^{-\alpha}, \text{ for } z \in \mathbb{C}, \;
Re(z) > -\min_{1\leq i \leq n} \frac{\alpha}{\nu_{i}(t)}.
\end{eqnarray*}
Differentiating with respect to $t$ yields
\begin{eqnarray*}
\frac{\partial J}{\partial t} (t,z) &=& J (t,z) \frac{\partial}{\partial t} (\ln(J)) \\
&=& J (t,z)\sum_{i=j,k} \frac{(\mu_{i} - \lambda_{i})z}{1 + \frac {\nu_{i}(t) z}{\alpha}}. 
\end{eqnarray*}
We take the inverse transform to get
\begin{eqnarray*}
\frac{\partial }{\partial t} F_{Y(t)}(x) &=& \sum_{i=j,k} (\mu_{i} - \lambda_{i}) \frac{\partial}{\partial x} \Pr( Y(t) + \nu_{i}(t) \psi_{i} \leq x) \\
&=& \sum_{i=j,k} (\mu_{i} - \lambda_{i})  f_{ Y(t) + \nu_{i}(t) \psi_{i}} (x) \\
&=& (\mu_{j} - \lambda_{j}) [f_{ Y(t) + \nu_{j}(t) \psi_{j}}(x) - f_{Y(t) + \nu_{k}(t) \psi_{k}}(x)],
\end{eqnarray*}
where $\psi_{j},\psi_{k} \sim Gamma(1,\alpha)$ are i.i.d gamma r.v's which are also independent of all $X_{i}$'s. By~\eqref{shcur_vec_2_componts}, we must have that $\lambda_{j} \geq \mu_{j}$, so it suffices to show that $ f_{ Y(t) + \nu_{j}(t) \psi_{j}}(x) \geq  f_{Y(t) + \nu_{k}(t) \psi_{k}}(x)$ for $x > (2  \alpha + 1)/(2 \alpha )$ and $ f_{ Y(t) + \nu_{j}(t) \psi_{j}}(x) \leq  f_{Y(t) + \nu_{k}(t) \psi_{k}}(x)$ for $x < (\alpha - 1)/\alpha$. On the other hand, using the Laplace transform and inverting it again, one can show the following identity (for any integer $k \geq 1$ and reals $a,b \geq 0$)
\begin{eqnarray}
\frac{d^{k-1}}{d x^{k-1}} f_{ X + a \psi_{1}} (x) -  \frac{d^{k-1}}{d x^{k-1}} f_{X + b \psi_{2}}(x) = \frac{1}{\alpha} (b-a) \frac{d^{k}}{d x^{k}} f_{X + a \psi_{1} + b \psi_{2}}(x),
\label{diff_f_gen}
\end{eqnarray}
where $X$ is  an arbitrary continuous positive r.v and $\psi_{i} \sim Gamma(1,\alpha) \;,\; i=1,2$, are i.i.d gamma r.v's which are 
also independent of $X$. As such we have
\begin{equation*}
f_{ Y(t) + \nu_{j}(t) \psi_{j}}(x) -  f_{Y(t) + \nu_{k}(t) \psi_{k}}(x) = \frac{1}{\alpha} (\nu_{k}(t) - \nu_{j}(t)) \frac{\partial}{\partial x} f_{Y(t) + \nu_{j}(t) \psi_{j} + \nu_{k}(t) \psi_{k}}(x).
\end{equation*}
For given $1 \leq j < k \leq n$, consider the r.v $Y(t) + \nu_{j}(t) \psi_{j} + \nu_{k}(t) \psi_{k}$. By~\eqref{schur_vect}, we get that $\nu_{j}(t) \geq \nu_{k}(t)$ for $t \in [0,1]$, hence, it is only left to show that, for any $t \in [0,1]$, the mode of this r.v, under the conditions 
\begin{eqnarray*}
&& \nu_{1}(t) \geq \ldots  \geq \nu_{n}(t) \geq 0, \\
&& \sum_{i=1}^{n} \nu_{i}(t) = 1,
\end{eqnarray*}
falls between $(\alpha-1)/\alpha$ and $(2 \alpha+1)/(2 \alpha)$. By Lemma~\ref{lemma_1*}, the mode of r.v $Y(t) + \nu_{1}(t) \psi_{1} + \nu_{2}(t) \psi_{2}$ is greater than that of $Y(t) + \nu_{j}(t) \psi_{j} + \nu_{k}(t) \psi_{k}$ for $1 < j < k$, and also the mode of $Y(t) + \nu_{n-1}(t) \psi_{n-1} + \nu_{n}(t) \psi_{n}$ is smaller than that of $Y(t) + \nu_{j}(t) \psi_{j} + \nu_{k}(t) \psi_{k}$ for $j < k < n$. Hence, we only need to show that for the mode of the r.v 
\begin{subequations}
\begin{equation}
Y_{1}(t) \defeq Y(t) + \nu_{1}(t) \psi_{1} + \nu_{2}(t) \psi_{2}
\label{Y_1}
\end{equation}
denoted by $\bar{x}_{1}(t)$, we have $\bar{x}_{1}(t) \leq (2 \alpha+1)/(2 \alpha)$ and for the mode of the r.v 
\begin{equation}
Y_{n}(t) \defeq Y(t) + \nu_{n-1}(t) \psi_{n-1} + \nu_{n}(t) \psi_{n}
\label{Y_n}
\end{equation}
\label{Y_1_Y_n}
\end{subequations}
denoted by $\bar{x}_{n}(t)$, we have $\bar{x}_{n}(t) \geq (\alpha-1)/\alpha$. 
In what follows, we fix any $t \in [0,1]$ and, for notational simplicity, denote $Y_{1}(t), \bar{x}_{1}(t), Y_{n}(t)$ and $\bar{x}_{n}(t)$ by $Y_{1}, \bar{x}_{1}, Y_{n}$ and $\bar{x}_{n}$, respectively. 

We first prove the case for $Y_{1}$. At any mode of $Y_{1}$, we have\footnote{Since mode is unique, thus f must be strictly concave at the mode} 
\begin{eqnarray*}
\left[ \frac{\partial}{\partial x} f_{Y_{1}}(x, \bm \nu) \right]_{({\bar{x}_{1}(\bm \nu), \bm \nu})} = 0, \\
\left[ \frac{\partial^{2}}{\partial x^{2}} f_{Y_{1}}(x, \bm \nu) \right]_{({\bar{x}_{1}(\bm \nu), \bm \nu})} < 0,
\end{eqnarray*}
hence implicit function theorem yields
\begin{equation}
\frac{\partial \bar{x}_{1}}{\partial \nu_{1}} (\bm \nu) = - \frac{\frac{\partial^{2} f_{Y_{1}}}{\partial \nu_{1}\partial \bar{x}_{1}} (\bar{x}_{1}(\bm \nu), \bm \nu)}{\frac{\partial^{2} f_{Y_{1}}}{\partial {\bar{x}_{1}}^{2}} (\bar{x}_{1}(\bm \nu), \bm \nu) }.
\label{d_bar_x_d_nu}
\end{equation}
Let $\bm \nu^{*}$ be where the global maximum of $\bar{x}_{1}(\bm \nu)$, denoted by $\bar{x}^{*}_{1}$, occurs. Thus by the necessary condition of maximality, we must have
\begin{equation*}
\left[ \frac{\partial^{2}}{\partial \nu_{1} \partial x} f_{Y_{1}}(x, \bm \nu) \right]_{(\bar{x}^{*}_{1},\bm \nu^{*})} = 0.
\end{equation*}

Suppose where the maximum occurs, i.e., $\bm \nu^{*}$, there is a nonzero coefficient, $\nu_{3}$, such that $0 < \nu_{3} < \nu_{2} \leq \nu_{1}$. The case of $\nu_{3} = 0$ will be dealt with at the end of the proof. Fixing all other coefficients, we vary $\nu_{2}$ and $\nu_{3}$ under the condition $\nu_{2} + \nu_{3} = const$. That is, we take the directional derivative in the direction $\delta\bm{\nu} = (0, 1, -1, 0 \ldots, 0)$ for the ``$n$'' dimensional vector $\bm{\nu} = (\nu_{1},\nu_{2},\nu_{3},\ldots,\nu_{n})$. In other words, we consider the change in the direction of $\bm{\nu} + \gamma \delta\bm{\nu}$. Using Laplace transform, we get
\begin{eqnarray*}
&&\frac{d}{d\gamma}\mathcal{L}[Y_{1}](z) = \mathcal{L}[Y_{1}](z) \frac{d}{d\gamma} \ln \left[ \mathcal{L}[Y_{1}](z) \right] \\
&=&  \mathcal{L}[Y_{1}](z) \frac{d}{d\gamma} \Big( - (1+\alpha) \ln (1 + \frac {\nu_{2} z}{\alpha}) -\alpha \sum_{\substack{i=0 \\ i\neq 2}}^{n} \ln (1 + \frac {\nu_{i} z}{\alpha}) - \ln (1 + \frac {\nu_{1} z}{\alpha}) \Big) \\
&=&  \mathcal{L}[Y_{1}](z) \left( - \frac{(1+\alpha) z}{\alpha} \frac{1}{(1 + \frac {\nu_{2} z}{\alpha})} +   \frac{z}{(1 + \frac {\nu_{3} z}{\alpha})} \right) \\
&=&  z \mathcal{L}[Y_{1}](z) \left( \frac{1}{(1 + \frac {\nu_{3} z}{\alpha})} -  \frac{1}{(1 + \frac {\nu_{2} z}{\alpha})} - \frac{1}{\alpha}\frac{1}{(1 + \frac {\nu_{2} z}{\alpha})} \right) \\
&=&  z \mathcal{L}[Y_{1}](z) \left( \frac{\frac{(\nu_{2} - \nu_{3}) z}{\alpha}}{(1 + \frac {\nu_{3} z}{\alpha})(1 + \frac {\nu_{2} z}{\alpha})} - \frac{1}{\alpha}\frac{1}{(1 + \frac {\nu_{2} z}{\alpha})} \right).
\end{eqnarray*}
Now inverting the above and differentiating (w.r.t\ $x$), we get
\begin{eqnarray}
\frac{\partial^{2}}{\partial \gamma \partial x} f_{Y_{1}} (x, \bm \nu) &=& \frac{(\nu_{2} - \nu_{3}) }{\alpha} \frac{\partial^{3}}{\partial x^{3}} f_{\tilde{Y}_{23}} (x, \bm \nu) - \frac{1}{\alpha} \frac{\partial^{2}}{\partial x^{2}} f_{Y_{1} + \nu_{2} \xi_{2}} (x, \bm \nu),
\label{partial_13_1}
\end{eqnarray}
where $\tilde{Y}_{23} = Y_{1} + \nu_{2} \xi_{2} + \nu_{3} \xi_{3}$ with $\xi_{2}, \xi_{3} \sim Gamma(1,\alpha)$ being  i.i.d r.v's, independent of all others appearing before. So at the maximum, $\bar{x}^{*}_{1}$, we must have
\begin{equation}
\left[ \frac{\partial^{2}}{\partial \gamma \partial x} f_{Y_{1}} (x, \bm \nu) \right]_{(\bar{x}^{*}_{1}, \bm \nu^{*})} = 0.
\label{partial_13_1_maximum}
\end{equation}
Now consider a slight perturbation of $\delta\bm{\nu} = (0, 1, -1, 0 \ldots, 0)$ as $\delta\bm{\nu}(\veps_{0}) = (0, 1, -(1 + \veps_{0}), 0 \ldots, 0)$ for some $\veps_{0} > 0$. Computations similar as above yields
\begin{eqnarray*}
\frac{\partial^{2}}{\partial \gamma \partial x} f_{Y_{1}} (x, \bm \nu) = \frac{(\nu_{2} - \nu_{3}) }{\alpha} \frac{\partial^{3}}{\partial x^{3}} f_{\tilde{Y}_{23}} (x, \bm \nu) &-& \frac{1}{\alpha} \frac{\partial^{2}}{\partial x^{2}} f_{Y_{1} + \nu_{2} \xi_{2}} (x, \bm \nu) \nonumber \\
&+& \veps_{0} \frac{\partial^{2}}{\partial x^{2}} f_{Y_{1} + \nu_{3} \xi_{3}} (x, \bm \nu).
\end{eqnarray*}
So at the maximum, $\bar{x}^{*}_{1}$, using~\eqref{partial_13_1_maximum}, we must have
\begin{eqnarray*}
\left[ \frac{\partial^{2}}{\partial \gamma \partial x} f_{Y_{1}} (x, \bm \nu) \right]_{(\bar{x}^{*}_{1}, \bm \nu^{*})} &=& \veps_{0} \left[ \frac{\partial^{2}}{\partial x^{2}} f_{Y_{1} + \nu_{3} \xi_{3}} (x, \bm \nu) \right]_{(\bar{x}^{*}_{1}, \bm \nu^{*})} \\
&=& - \frac{\alpha \veps_{0}}{\nu_{3}} \left[ \frac{\partial}{\partial x} f_{Y_{1} + \nu_{3} \xi_{3}} (x, \bm \nu) \right]_{(\bar{x}^{*}_{1}, \bm \nu^{*})} < 0.
\end{eqnarray*}
This followed from~\eqref{diff_f_gen} (with $k=2$, $a = 0$ and $b = \nu_{3}$) and Lemma~\ref{lemma_1*}, as well as noting that $\left[ \frac{\partial}{\partial x} f_{Y_{1}} (x, \bm \nu)\right]_{(\bar{x}^{*}_{1}, \bm \nu^{*} )}= 0$. Using~\eqref{d_bar_x_d_nu}, the previous relation implies that the mode must increase if a sufficiently small step is taken along the perturbed direction $\delta\bm{\nu}(\veps_{0}) = (0, -1, (1 + \veps_{0}), 0 \ldots, 0)$. 
In other words, for $\veps_{1} >0$ small enough, for the mode of r.v 
\begin{equation}
Y^{(\veps_{0},\veps_{1})}_{1} \defeq \nu_{1}  X_{1} + (\nu_{2} -\veps_{1}) X_{2} + \big(\nu_{3} + \veps_{1} (1 + \veps_{0}) \big) X_{3} + \sum_{i = 4}^{n} \nu_{i} X_{i} + \nu_{1} \psi_{1} + (\nu_{2} - \veps_{1}) \psi_{2},
\label{y_veps_1}
\end{equation}
denoted by $\bar{x}^{(\veps_{0},\veps_{1})}_{1}$, we must have that 
\begin{equation}
\bar{x}^{*}_{1} < \bar{x}^{(\veps_{0},\veps_{1})}_{1}.
\label{ineq_mode_veps_1}
\end{equation}
It is clear that $\veps_{1}$ depends on $\veps_{0}$ and consequently, as $\veps_{0}$ gets smaller, $\veps_{1}$ might get smaller as well. However, $\veps_{1} \veps_{0} \in o(\veps_{0})$ and $\veps_{1} \veps_{0} \in o(\veps_{1})$, where ``$o$'' denotes the ``\textit{little o}''. In other words, consider decreasing sequences of $(\veps_{0}^{n})_{n=1}^{\infty}$ and $(\veps^{n}_{1})_{n=1}^{\infty}$. Since $\lim_{n \rightarrow \infty} (\veps_{1}^{n} \veps_{0}^{n})/ \veps^{n}_{0} = \lim_{n \rightarrow \infty} (\veps_{1}^{n} \veps^{n}_{0})/ \veps_{1}^{n} = 0$, it follows that, as $\veps_{0}$ gets smaller, $\veps_{1} \veps_{0}$ term in~\eqref{y_veps_1} vanishes at a faster rate than either of $\veps_{0}$ or $\veps_{1}$. Hence 
using~\eqref{y_veps_1} and~\eqref{ineq_mode_veps_1}, as well as the continuity of $\bar{x}_{1}$, we can consider a r.v 
\begin{equation*}
Y^{\veps}_{1} \defeq \nu_{1} X_{1} + (\nu_{2}  - \veps) X_{2} + (\nu_{3} + \veps) X_{3} + \sum_{i = 4}^{n} \nu_{i} X_{i} + \nu_{1}  \psi_{1} + (\nu_{2} -\veps) \psi_{2},
\end{equation*}
which has the same form as $Y_{1}$ but whose mode, denoted by $\bar{x}^{\veps}_{1}$, satisfies
\begin{equation*}
\bar{x}^{*}_{1} < \bar{x}^{\veps}_{1},
\end{equation*}
which contradicts the maximality of $\bar{x}^{*}_{1}$. This means that, if $\nu_{3} \neq 0$, then the maximum must occurs at $\bm \nu^{*}$ with $\nu_{3} = \nu_{2} \leq \nu_{1}$. Similar arguments show that at the maximum, we must have $\nu_{n} = \ldots = \nu_{3} = \nu_{2} \leq \nu_{1}$. But by~\eqref{partial_13_1} and the fact that
\begin{equation*}
\left[ \frac{\partial^{2}}{\partial x^{2}} f_{Y_{1} + \nu_{2} \xi_{2}} (x, \bm \nu) \right]_{(\bar{x}^{*}_{1}, \bm \nu^{*})} = - \frac{\alpha}{\nu_{2}} \left[ \frac{\partial}{\partial x} f_{Y_{1} + \nu_{2} \xi_{2}} (x, \bm \nu) \right]_{(\bar{x}^{*}_{1}, \bm \nu^{*})} < 0,
\end{equation*}
we must have that $\nu_{3} \neq \nu_{2}$. As such, the only possibility is $\nu_{n} = \ldots = \nu_{3} = 0$, i.e., the maximum mod of $Y_{1}$ coincides with that of r.v $\zeta = \nu \zeta_{1} + (1-\nu) \zeta_{2}$ where $\zeta_{1}, \zeta_{2} \sim Gamma(1+\alpha,\alpha)$. So now lemma~\ref{lemma_2*} gives 
\begin{equation*}
\bar{x}^{*}_{1} \leq \frac{2 \alpha + 1}{2 \alpha }.
\end{equation*}

Now consider the case for $Y_{n}$. Suppose where the minimum occurs, i.e., $\bm \nu^{*}$, we have, $\nu_{n-1} > 0$ and $\nu_{n-2} > \nu_{n-1}$. Now by the same reasoning as in the case of $Y_{1}$, we can consider a r.v 
\begin{equation*}
Y^{\veps}_{n} \defeq \sum_{i = 1}^{n-3} \nu_{i} X_{i} \nu_{1} X_{1} + (\nu_{n-2}  - \veps) X_{n-2} + (\nu_{n-1} + \veps) X_{n-1} + \nu_{n} X_{n} + (\nu_{n-1} + \veps)  \psi_{n-1} + \nu_{n} \psi_{n},
\end{equation*}
which has the same form as $Y_{n}$ but whose mode, denoted by $\bar{x}^{\veps}_{n}$, satisfies
\begin{equation*}
\bar{x}^{\veps}_{n} < \bar{x}^{*}_{n},
\end{equation*}
which contradicts the minimality of $\bar{x}^{*}_{n}$. This means that, if $\nu_{n-1} \neq 0$, then the minimum must occurs at $\bm \nu^{*}$ with $\nu_{n-2} = \nu_{n-1}$. But again, as in the end of the proof for the case of $Y_{1}$, we must have that $\nu_{n-2} \neq \nu_{n-1}$. As such the only possibility is $\nu_{n-1} = 0$ (which would also mean that $\nu_{n} = 0$, although this case can also be independently established with the same reasoning as above). Hence, the minimum mode of $Y_{n}$ must coincide with that of the r.v $\sum_{i = 1}^{n-2} \nu_{i} X_{i}$. Again, appealing to the same line of reasoning, at the minimum mode of $\sum_{i = 1}^{n-2} \nu_{i} X_{i}$, we must either have $\nu_{j} = 0$ or $\nu_{j-1} = \nu_{j}$ for any $1 < j \leq n-2$. Hence, for $n \geq 3$, we get
\begin{equation*}
\bar{x}^{*}_{n} = \min_{1 \leq j \leq n-2} \bar{x}_{j \alpha} = \frac{\alpha - 1}{\alpha},  
\end{equation*}
where $\bar{x}_{j \alpha}$ denotes the mode of the r.v $Gamma\big(j\alpha, j \alpha \big)$.

The case of $n=2$, amounts to studying the mode of r.v $\zeta = \nu \zeta_{1} + (1-\nu) \zeta_{2}$, denoted by $\bar{x}(\zeta)$, where $\zeta_{1}, \zeta_{2} \sim Gamma(1+\alpha,\alpha)$. Direct application of Lemma~\ref{lemma_2*} yields $1 \leq \bar{x}(\zeta) \leq (2 \alpha + 1)/(2 \alpha)$. Theorem~\ref{majorization_thm} is proved.
$\blacksquare$

\subsection{Proof of Corollary~\ref{majorization_cor_weak}}
\label{proof_majorization_cor_weak} 
The proof goes along the same line as that of~Theorem~\ref{majorization_thm}. For $t \in [0,1]$, we again define 
\begin{eqnarray*}
\nu_{i}(t) &\defeq& t \lambda_{i} + (1-t) \mu_{i}, \quad i = j,k,\\
\nu_{i}(t) &\defeq& \lambda_{i} , \quad i \neq j,k,\\
Y(t) &\defeq& \sum_{i = 1}^{n} \nu_{i}(t) X_{i}.
\end{eqnarray*}
where 
\begin{eqnarray*}
&& \nu_{1}(t) \geq \ldots  \geq \nu_{n}(t) \geq 0, \quad t \in [0,1], \\
&& \sum_{i=1}^{n} \nu_{i}(t) = s_{\bm{\mu}} + t (s_{\bm{\lambda}} - s_{\bm{\mu}}).
\end{eqnarray*}
Computations identical to the proof of Theorem~\ref{majorization_thm} give
\begin{eqnarray*}
\bar{x}^{*}_{1}(t) &\leq& \big( s_{\bm{\mu}} + t (s_{\bm{\lambda}} - s_{\bm{\mu}}) \big) \frac{2 \alpha + 1}{2 \alpha }, \\
\bar{x}^{*}_{n}(t) &\geq& \big( s_{\bm{\mu}} + t (s_{\bm{\lambda}} - s_{\bm{\mu}}) \big) \frac{\alpha - 1}{\alpha },
\end{eqnarray*}
where $\bar{x}^{*}_{1}(t)$ and $\bar{x}^{*}_{n}(t)$ are, respectively, the maximum and minimum mode of r.v's $Y_{1}(t)$ and $Y_{n}(t)$ which are defined similarly as in~\eqref{Y_1_Y_n}.
Now for $t \in [0,1]$, we get
\begin{eqnarray*}
\bar{x}^{*}_{1}  &=& \max_{t \in [0,1]} \bar{x}^{*}_{1}(t) \leq \frac{( 2 \alpha + 1 ) s_{\bm{\lambda}} }{2 \alpha }, \\
\bar{x}^{*}_{n}  &=& \min_{t \in [0,1]} \bar{x}^{*}_{n}(t) \geq \frac{( \alpha - 1 ) s_{\bm{\mu}} }{\alpha },
\end{eqnarray*}
which give the desired results.
$\blacksquare$

\subsection{Proof of Corollary~\ref{majorization_cor_weak_infty}}
\label{proof_majorization_cor_weak_infty}
Define
\begin{equation*}
P_{n}(\bm{\lambda}_{n};\alpha,\beta,x) \defeq \Pr \left( \sum_{i=1}^{n} \lambda_{i} X_{i} < x \right),
\end{equation*}
where
\begin{equation*}
\bm{\lambda}_{n} \defeq (\lambda_{1},\lambda_{2},\ldots,\lambda_{n}).
\end{equation*}
By the continuity from above, it is clear that
\begin{equation*}
P_{\infty}(\bm{\lambda};\alpha,\beta,x) = \lim_{n \rightarrow \infty} P_{n}(\bm{\lambda}_{n};\alpha,\beta,x).
\end{equation*}
Now since $\bm{\mu}_{n} \prec_{w} \bm{\lambda}_{n}$, Corollary~\ref{majorization_cor_weak} yields
\begin{eqnarray*}
P_{n}(\bm{\mu};\alpha,\beta,x) &\geq& P_{n}(\bm{\lambda};\alpha,\beta,x), \quad \forall x > \frac {(2  \alpha + 1) \sum_{i=1}^{n} \lambda_{i} }{2 \beta }, \\
P_{n}(\bm{\mu};\alpha,\beta,x) &\leq& P_{n}(\bm{\lambda};\alpha,\beta,x), \quad \forall x <
  	\frac{(\alpha-1) \sum_{i=1}^{n} \mu_{i}}{\beta} , \quad \alpha > 1.
\end{eqnarray*}
Taking the limit as $n \rightarrow \infty$ gives the desired result.
$\blacksquare$

%


\end{document}

%% file: exline.tex
%
%
\vspace{0.5cm}